\documentclass[]{interact}

\usepackage{epstopdf}
\usepackage{subfigure}

\usepackage[natbibapa,nodoi]{apacite}
\theoremstyle{plain}
\newtheorem{theorem}{Theorem}[section]
\newtheorem{lemma}[theorem]{Lemma}

\theoremstyle{definition}

\theoremstyle{remark}
\newtheorem{remark}{Remark}

\def\R{\mathbb{R}}

\begin{document}


\title{Boundary Output Feedback Stabilization of Reaction-Diffusion PDEs with Delayed Boundary Measurement}

\author{
\name{Hugo Lhachemi\textsuperscript{a}\thanks{Email: hugo.lhachemi@centralesupelec.fr} and Christophe Prieur\textsuperscript{b}}
\affil{\textsuperscript{a}Universit{\'e} Paris-Saclay, CNRS, CentraleSup{\'e}lec, Laboratoire des signaux et syst{\`e}mes, 91190, Gif-sur-Yvette, France; \textsuperscript{b}Universit{\'e} Grenoble Alpes, CNRS, Grenoble-INP, GIPSA-lab, F-38000, Grenoble, France}
}

\maketitle

\begin{abstract}
This paper addresses the boundary output feedback stabilization of general 1-D reaction-diffusion PDEs with delayed boundary measurement. The output takes the form of a either Dirichlet or Neumann trace. The output delay can be arbitrarily large. The control strategy is composed of a finite-dimensional observer that is used to observe a delayed version of the first modes of the PDE and a predictor component which is employed to obtain the control input to be applied  at current time. For any given value of the output delay, we assess the stability of the resulting closed-loop system provided the order of the observer is selected large enough. Taking advantage of this result, we discuss the extension of the control strategy to the case of simultaneous input and output delays.
\end{abstract}

\begin{keywords}
Reaction-diffusion PDEs, output feedback, delayed measurement, boundary control.
\end{keywords}

\section{Introduction}\label{sec: Introduction}

Time delays commonly arise in the design of control strategies due to either natural feedback processes or the active implementation of control laws. Moreover, time delays are well-known for their capability to introduce instabilities when not considered properly in the control design. For these reasons, the feedback control of finite-dimensional systems in the presence of delays has been extensively studied \citep{artstein1982linear,richard2003time}. The extension of this problematic to Partial Differential Equations (PDEs) has been the topic of a number of papers in the recent years \citep{nicaise2008stabilization,wang2018delay}. In particular, the development of control strategies for the feedback stabilization of reaction-diffusion PDEs with an arbitrarily long delay in the either control input \citep{katz2021sub,krstic2009control,lhachemi2019lmi,lhachemi2020feedback,lhachemi2021robustness,lhachemi2021predictor,qi2020compensation} or state \citep{hashimoto2016stabilization,kang2017boundary,lhachemi2020boundary,lhachemi2021boundary} has been intensively studied.

In this paper, we address the boundary output feedback stabilization of general 1-D reaction-diffusion PDEs with delayed boundary measurement. The control input and boundary conditions take the form of Dirichlet/Neumann/Robin boundary conditions. The output is selected as a either Dirichlet or Neumann boundary trace presenting an arbitrarily long delay. The control strategy couples a finite-dimensional observer \citep{curtain1982finite,balas1988finite,harkort2011finite,grune2021finite} used to observe a finite number of modes of the PDE and a predictor component \citep{artstein1982linear,karafyllis2017predictor}. To design the finite-dimensional observer, we leverage the approach reported first in \citep{katz2020constructive} relying on spectral-reduction methods \citep{russell1978controllability,coron2004global,coron2006global}, and more specifically on the scaling-based procedures described in \citep{lhachemi2020finite,lhachemi2021nonlinear} that allow to handle Dirichlet/Neumann boundary measurement while performing, for very general 1-D reaction-diffusion PDEs, the control design directly with the actual control input $u$ and not its time-derivative $v = \dot{u}$. We refer the reader, e.g., to~\citep[Sec.~3.3.]{curtain2012introduction} for a general introduction to the topic of boundary control systems. 

It is worth noting that the robustness of the finite-dimensional control strategy reported in \citep{katz2020constructive} to small enough input and measurement delays was discussed in \citep{katz2021delayed}. However, the presence of an arbitrarily long output delay imposes more stringent constraints on the system and requires the development of a dedicated control strategy. This is achieved in this paper by leveraging a predictor design \citep{karafyllis2017predictor,deng2019prediction}. The possibility to couple a finite-dimensional observer with a predictor to handle arbitrary input delays was reported first in \citep{katz2021sub} in the very specific configuration of a Neumann boundary control, for a bounded output operator, and for system trajectories evaluated in $L^2$ norm. The case of general input delayed 1-D reaction-diffusion PDEs with Dirichlet/Neumann/Robin boundary control and Dirichlet/Neumann boundary measurement was solved in \citep{lhachemi2021predictor} for PDE trajectories in $H^1$ norm. In this paper we address the dual problem of \citep{lhachemi2021predictor}, namely the output feedback stabilization of reaction-diffusion PDEs in the presence of an arbitrary output delay. The proposed control strategy is composed of a finite-dimensional observer that is used to observe a delayed version of the first modes of the PDE (this delayed observation matches with the measurement delay) and a predictor component which is employed to obtain the control input to be applied  at current time. For a given value of the output delay, we derive a set of sufficient LMI conditions ensuring the exponential stability of the resulting closed-loop system for PDE trajectories evaluated in $H^1$ norm. For any given value of the output delay, these control design constraints are shown to be feasible provided the order of the observer is selected large enough. Combining the approach developed in this paper with the one reported in \citep{lhachemi2021predictor} for the case of an input delay, we also discuss the extension of the method to the stabilization of reaction-diffusion PDEs in the presence of both input and output delays.

The paper is organized as follows. After introducing some definitions and properties, the control design problem addressed in this paper is presented in Section~\ref{sec: preliminaries}. The case of a delayed Dirichlet boundary measurement is reported in Section~\ref{sec: Dirichlet}. The control design procedure is then extended to delayed Neumann boundary measurement in Section~\ref{sec: Neumann}. A numerical illustration of these two settings is presented in Section~\ref{sec: numerics}. The extension of the obtained results to the case of input and output delays is discussed in Section~\ref{sec: extension}. Finally, concluding remarks are formulated in Section~\ref{sec: conclusion}.

\section{Definitions and problem setting}\label{sec: preliminaries}

\subsection{Definitions and properties}

\subsubsection{Notation}

Spaces $\R^n$ are equipped with the Euclidean norm denoted by $\Vert\cdot\Vert$. The associated induced norms of matrices are also denoted by $\Vert\cdot\Vert$. For any two vectors $X$ and $Y$ of arbitrary dimensions, $ \mathrm{col} (X,Y)$ stands for the vector $[X^\top,Y^\top]^\top$. $L^2(0,1)$ stands for the space of square integrable functions on $(0,1)$ and is endowed with the inner product $\langle f , g \rangle = \int_0^1 f(x) g(x) \,\mathrm{d}x$. The corresponding norm is denoted by $\Vert \cdot \Vert_{L^2}$. For an integer $m \geq 1$, $H^m(0,1)$ stands for the $m$-order Sobolev space and is endowed with its usual norm $\Vert \cdot \Vert_{H^m}$. For any symmetric matrix $P \in\R^{n \times n}$, $P \succeq 0$ (resp. $P \succ 0$) indicates that $P$ is positive semi-definite (resp. positive definite).

\subsubsection{Properties of Sturm-Liouville operators}

Let $\theta_1,\theta_2\in[0,\pi/2]$, $p \in \mathcal{C}^1([0,1])$ and $q \in \mathcal{C}^0([0,1])$ with $p > 0$ and $q \geq 0$. Let the Sturm-Liouville operator $\mathcal{A} : D(\mathcal{A}) \subset L^2(0,1) \rightarrow L^2(0,1)$ be defined by $\mathcal{A}f = - (pf')' + q f$ on the domain $D(\mathcal{A}) = \{ f \in H^2(0,1) \,:\, c_{\theta_1} f(0) - s_{\theta_1} f'(0) = c_{\theta_2} f(1) + s_{\theta_2} f'(1) = 0 \}$. Here we use the short notations $c_{\theta_i} = \cos\theta_i$ and $s_{\theta_i} = \sin\theta_i$. It is well-known that the eigenvalues $\lambda_n$, $n \geq 1$, of $\mathcal{A}$ are simple, non negative, and form an increasing sequence with $\lambda_n \rightarrow + \infty$ as $n \rightarrow + \infty$. Moreover the corresponding unit eigenvectors $\phi_n \in L^2(0,1)$ form a Hilbert basis. The domain of the operator $\mathcal{A}$ is equivalently characterized in terms of the above eigenstructures by $D(\mathcal{A}) = \{ f \in L^2(0,1) \,:\, \sum_{n\geq 1} \vert \lambda_n \vert ^2 \vert \left< f , \phi_n \right> \vert^2 < +\infty \}$. Introducing $p_*,p^*,q^* \in \R$ so that $0 < p_* \leq p(x) \leq p^*$ and $0 \leq q(x) \leq q^*$ for all $x \in [0,1]$, we have 
$
0 \leq \pi^2 (n-1)^2 p_* \leq \lambda_n \leq \pi^2 n^2 p^* + q^*
$
for all $n \geq 1$ (see, e.g., \cite{orlov2017general}). Furthermore, with the additional assumption $p \in \mathcal{C}^2([0,1])$, we also have that $\phi_n (\xi) = O(1)$ and $\phi_n' (\xi) = O(\sqrt{\lambda_n})$ as $n \rightarrow + \infty$ for any given $\xi \in [0,1]$ (see, e.g., \cite{orlov2017general}). Besides and under the assumption $q > 0$, an integration by parts and the continuous embedding $H^1(0,1) \subset L^\infty(0,1)$ show the existence of constants $C_1,C_2 > 0$ such that
\begin{align}
C_1 \Vert f \Vert_{H^1}^2 \leq 
\sum_{n \geq 1} \lambda_n \left< f , \phi_n \right>^2
= \left< \mathcal{A}f , f \right>
\leq C_2 \Vert f \Vert_{H^1}^2 \label{eq: inner product Af and f}
\end{align}
for all $f \in D(\mathcal{A})$. The latter inequalities and the Riesz-spectral nature of $\mathcal{A}$ imply that the series expansion $f = \sum_{n \geq 1} \left< f , \phi_n \right> \phi_n$ holds in $H^2(0,1)$ norm for any $f \in D(\mathcal{A})$. Invoking again the continuous embedding $H^1(0,1) \subset L^{\infty}(0,1)$, we deduce that $f(0) = \sum_{n \geq 1} \left< f , \phi_n \right> \phi_n(0)$ and $f'(0) = \sum_{n \geq 1} \left< f , \phi_n \right> \phi_n'(0)$. 

In the sequel, we define for any integer $N \geq 1$ and any $f \in L^2(0,1)$ the quantity $\mathcal{R}_N f = \sum_{n \geq N+1} \left< f , \phi_n \right> \phi_n$.

\subsection{Problem setting and spectral reduction}

\subsubsection{Problem setting}

Let the reaction-diffusion system with boundary control be described by
\begin{subequations}\label{eq: PDE}
\begin{align}
& z_t(t,x) = (p(x) z_x(t,x))_x - \tilde{q}(x) z(t,x) \\
& c_{\theta_1} z(t,0) - s_{\theta_1} z_x(t,0) = 0 \\
& c_{\theta_2} z(t,1) + s_{\theta_2} z_x(t,1) = u(t) \\
& z(0,x) = z_0(x) 
\end{align}
\end{subequations}
for $t > 0$ and $x \in (0,1)$. Here $\theta_1 , \theta_2 \in [0,\pi/2]$, $p \in\mathcal{C}^2([0,1])$ with $p > 0$, and $\tilde{q} \in\mathcal{C}^0([0,1])$. The state of the reaction-diffusion PDE at time $t$ is $z(t,\cdot)$, the command is $u(t)$, and the initial condition is $z_0$. We define the initial command as $u_0 = c_{\theta_2} z_0(1) + s_{\theta_2} z_{0,x}(1)$. For some measurement delay $h > 0$, the system output is chosen as the either delayed Dirichlet or delayed Neumann trace. More precisely, in the case $\theta_1 \in (0,\pi/2]$, the delayed Dirichlet boundary measurement is defined by 
\begin{equation}\label{eq: system output Dirichlet}
y_D(t) = 
\left\{
\begin{array}{ll}
z(t-h,0) , & t \geq h \\
y_0(t-h) , & 0 \leq t \leq h
\end{array}
\right.
\end{equation}
Similarly but in the case $\theta_1 \in [0,\pi/2)$, the delayed Neumann boundary measurement is defined by 
\begin{equation}\label{eq: system output Neumann}
y_N(t) = 
\left\{
\begin{array}{ll}
z_x(t-h,0) , & t \geq h \\
y_0(t-h) , & 0 \leq t \leq h
\end{array}
\right.
\end{equation}
In both cases, $y_0 : [-h,0] \rightarrow \R$ is the initial condition of the delayed boundary measurement and is assumed to be Lipschitz continuous.

Without loss of generality, we introduce $q \in\mathcal{C}^0([0,1])$ and $q_c \in\R$ so that
\begin{equation}\label{eq: writting of tilde_q}
\tilde{q}(x) = q(x) - q_c , \qquad q(x) > 0  .
\end{equation}

\subsubsection{Spectral reduction}

Based on the change of variable formula
\begin{equation}\label{eq: change of variable}
w(t,x) = z(t,x) - \frac{x^2}{c_{\theta_2} + 2 s_{\theta_2}} u(t) 
\end{equation}
the PDE (\ref{eq: PDE}) in original coordinates can be equivalently reformulated as the homogeneous PDE described by 
\begin{subequations}\label{eq: PDE Dirichlet - homogeneous}
\begin{align}
& v(t) = \dot{u}(t) \\
& w_t(t,x) = (p(x) w_x(t,x))_x - \tilde{q}(x) w(t,x) + a(x) u(t) + b(x) v(t) \\ 
& c_{\theta_1} w(t,0) - s_{\theta_1} w_x(t,0) = 0 \\
& c_{\theta_2} w(t,1) + s_{\theta_2} w_x(t,1) = 0 \\
& w(0,x) = w_0(x) 
\end{align}
\end{subequations}
Here we have $a(x) = \frac{1}{c_{\theta_2} + 2 s_{\theta_2}} \{ 2p(x) + 2xp'(x) - x^2 \tilde{q}(x) \}$, $b(x) = -\frac{x^2}{c_{\theta_2} + 2 s_{\theta_2}}$, and $w_0(x) = z_0(x) - \frac{x^2}{c_{\theta_2} + 2 s_{\theta_2}} u(0)$. Noting that $w(t,0) = z(t,0)$ and $w_x(t,0) = z_x(t,0)$, the boundary measurements are described for $t \geq h$ by
\begin{equation}\label{eq: measurement homogeneous coordinates}
y_D(t) = w(t-h,0) , \qquad y_N(t) = w_x(t-h,0) .
\end{equation}

Let us now define the coefficients of projection $z_n(t) = \left< z(t,\cdot) , \phi_n \right>$, $w_n(t) = \left< w(t,\cdot) , \phi_n \right>$, $a_n = \left< a , \phi_n \right>$, and $b_n = \left< b , \phi_n \right>$. Owing to (\ref{eq: change of variable}), we infer that that
\begin{equation}\label{eq: link z_n and w_n}
w_n(t) = z_n(t) + b_n u(t), \quad n \geq 1 .
\end{equation}
We now project the two PDEs representations (\ref{eq: PDE}) and (\ref{eq: PDE Dirichlet - homogeneous}) into the Hilbert basis $(\phi_n)_{n \geq 1}$. The former representation gives
\begin{equation}\label{eq: dynamics z_n}
\dot{z}_n(t) = (-\lambda_n + q_c) z_n(t) + \beta_n u(t)
\end{equation}
where $\beta_n = a_n + (-\lambda_n+q_c)b_n = p(1) \{ - c_{\theta_2} \phi_n'(1) + s_{\theta_2} \phi_n(1) \} = O(\sqrt{\lambda_n})$. The latter representation implies that
\begin{subequations}\label{eq: dynamics w_n}
\begin{align}
\dot{u}(t) & = v(t) \\
\dot{w}_n(t) & = (-\lambda_n + q_c) w_n(t) + a_n u(t) + b_n v(t) 
\end{align}
\end{subequations}
Finally the delayed measurements (\ref{eq: measurement homogeneous coordinates}) can be expressed for $t \geq h$ as the following series expansions:
\begin{equation}\label{eq: system output - series}
y_D(t) = \sum_{n \geq 1} w_n(t-h) \phi_n(0) , \qquad
y_N(t) = \sum_{n \geq 1} w_n(t-h) \phi_n'(0) .
\end{equation}

\section{Case of a delayed Dirichlet measurement}\label{sec: Dirichlet}

We address in this section the output feedback stabilization of the reaction-diffusion PDE described by (\ref{eq: PDE}) for $\theta_1 \in (0,\pi/2]$ with delayed Dirichlet measurement (\ref{eq: system output Dirichlet}).

\subsection{Control strategy}

Let $\delta > 0$ and $N_0 \geq 1$ be such that $-\lambda_n + q_c < - \delta < 0$ for all $n \geq N_0 + 1$. Let $N \geq N_0 + 1$ be arbitrarily fixed and that will be specified later. Inspired by \cite[Chap.~3]{karafyllis2017predictor} in the context of finite-dimensional systems, we first design an observer that is used to estimate from the delayed measurement $y_D(t)$ the $N$ first modes $z_n(t-h)$ of the PDE at time $t-h$. The observer dynamics reads, for $t \geq 0$,  
\begin{subequations}\label{eq: controller - Dirichlet}
\begin{align}
\hat{w}_n(t) & = \hat{z}_n(t) + b_n u(t-h) \\
\dot{\hat{z}}_n(t) & = (-\lambda_n+q_c) \hat{z}_n(t) + \beta_n u(t-h) \label{eq: controller 1 - Dirichlet} \\
& \phantom{=}\; - l_n \left\{ \sum_{k=1}^N \hat{w}_k(t) \phi_k(0) - y_D(t) \right\}  ,\; 1 \leq n \leq N_0 \nonumber \\
\dot{\hat{z}}_n(t) & = (-\lambda_n+q_c) \hat{z}_n(t) + \beta_n u(t-h) ,\; N_0+1 \leq n \leq N \label{eq: controller 2 - Dirichlet}
\end{align}
\end{subequations}
where $l_n \in\R$ are the observer gains and with $u(\tau) = u_0$ for $\tau \leq 0$. So $\hat{z}_n(t)$ is seen as the estimation of $z_n(t-h)$ for times $t \geq h$. Note that no control input is actually applied to the system (\ref{eq: PDE}) in negative time. The definition of $u$ in negative time is only introduced here in order to make sure that the dynamics (\ref{eq: controller - Dirichlet}) is well-defined for all $t \geq 0$.

Since the observer (\ref{eq: controller - Dirichlet}) estimates the first modes of the PDE at time $t-h$ while the feedback must be applied at current time $t$, we need to introduce a predictor component. Defining $\hat{Z}^{N_0} = \begin{bmatrix} \hat{z}_1 & \ldots & \hat{z}_{N_0} \end{bmatrix}^\top$ along with $A_0 = \mathrm{diag}(-\lambda_1+q_c,\ldots,-\lambda_{N_0}+q_c)$ and $\mathfrak{B}_0 = \begin{bmatrix} \beta_1 & \ldots & \beta_{N_0} \end{bmatrix}^\top$, we introduce the following Artstein tranformation:
\begin{equation}\label{eq: controller - Dirichlet - Artstein}
\hat{Z}_A^{N_0} (t) = e^{A_0 h}  \hat{Z}^{N_0}(t) + \int_{t-h}^t e^{A_0(t-s)} \mathfrak{B}_0 u(s) \,\mathrm{d}s.
\end{equation}
We can now define the control input as 
\begin{equation}\label{eq: command input - Dirichlet}
u(t) = K \hat{Z}_A^{N_0}(t) 
\end{equation}
for all $t \geq 0$ where $K\in\R^{1 \times N_0}$ is the feedback gain. 

\begin{remark}
The controller described by (\ref{eq: controller - Dirichlet}-\ref{eq: command input - Dirichlet}) takes a form similar to the one reported in \citep{lhachemi2021predictor} in the case of an input delay. However, due to the output delay considered in this paper, the measurement $y_D(t)$ appearing in (\ref{eq: controller 1 - Dirichlet}) is a time delayed version of the Dirichlet trace as described by (\ref{eq: system output Dirichlet}). Moreover, the delayed input $u(t-h)$ appearing in (\ref{eq: controller - Dirichlet}) is not reminiscent of an actual input delay, as the ones considered in \citep{lhachemi2021predictor}, but is due to the fact that $\hat{z}_n(t)$ does not estimate $z_n(t)$ but $z_n(t-h)$ for $t \geq h$, so that the measurement $y_D(t) = w(t-h,0)$ can indeed be used to design a classical Luenberger observer. We refer to \cite[Chap.~3]{karafyllis2017predictor} for general explanations of such a control design strategy in the context of output delayed finite-dimensional systems.
\end{remark}

\begin{remark}
Equations (\ref{eq: controller - Dirichlet - Artstein}-\ref{eq: command input - Dirichlet}) imply that the initial condition $\hat{Z}^{N_0}(0)\in\R^{N_0}$ of the $N_0$ first modes of the observer must be selected so that $u_0 = K \hat{Z}_A^{N_0}(0)$. For a given $u_0 \in \R$, the latter condition is equivalent to $K e^{A_0 h} \hat{Z}^{N_0}(0) = \left( 1 - \int_{-h}^0 K e^{- A_0 s} \mathfrak{B}_0 \,\mathrm{d}s \right) u_0$. This is possible as soon as $K \neq 0$.
\end{remark}

\begin{remark}\label{rmk: WP}
The well-posedness of the closed-loop system composed of the plant (\ref{eq: PDE}), the delayed Dirichlet measurement (\ref{eq: system output Dirichlet}), and the controller (\ref{eq: controller - Dirichlet}-\ref{eq: command input - Dirichlet}), is not trivial under this form due to the integral term $\varphi(t) = \int_{t-h}^{t} e^{A_0(t-s)} \mathfrak{B}_0 u(s) \,\mathrm{d}s$ appearing in (\ref{eq: controller - Dirichlet - Artstein}). However, it is observed that such a function $\varphi$ is the unique solution to the EDO 
\begin{equation}\label{eq: EDO varphi}
\dot{\varphi}(t) = A_0 \varphi(t) + \mathfrak{B}_0 K \hat{Z}_A^{N_0}(t) - e^{A_0 h} \mathfrak{B}_0 u(t-h)
\end{equation}
associated with the initial condition $\varphi(0) = \int_{-h}^{0} e^{-A_0 s} \mathfrak{B}_0 u_0 \,\mathrm{d}s$. Hence considering the infinite-dimensional system described by the plant (\ref{eq: PDE}), the delayed Dirichlet measurement (\ref{eq: system output Dirichlet}), the observer dynamics (\ref{eq: controller - Dirichlet}), the control input (\ref{eq: command input - Dirichlet}) with $\hat{Z}_A^{N_0}(t)$ define by 
\begin{equation*}
\hat{Z}_A^{N_0}(t) = e^{A_0 h} \hat{Z}^{N_0}(t) + \varphi(t) ,
\end{equation*}
along with the ODE (\ref{eq: EDO varphi}), the well-posedness in terms of classical solutions for initial conditions $z_0 \in H^2(0,1)$ and $\hat{z}_n(0) \in\R$ so that $c_{\theta_1} z_0(0) - s_{\theta_1} z_0'(0) = 0$ and $c_{\theta_2} z_0(1) + s_{\theta_2} z_0'(1) = u_0 = K \hat{Z}_A^{N_0}(0)$, and any Lipschitz continuous $y_0 \in \mathcal{C}^0([-h,0])$ so that $y_0(0) = z_0(0)$, is now an immediate consequence of \cite[Thm.~6.3.1 and~6.3.3]{pazy2012semigroups} and the use of a classical induction argument.
\end{remark}

\subsection{Truncated model for stability analysis}\label{subsec: truncated model}

In order to complete the tuning of the controller gains and to perform the stability analysis, we need to introduce first a finite dimensional model capturing the $N$ first modes of the PDE in $z$ coordinates (\ref{eq: PDE}) and the controller dynamics (\ref{eq: controller - Dirichlet}-\ref{eq: command input - Dirichlet}) based on the delayed Dirichlet measurement (\ref{eq: system output Dirichlet}). To do so we define the observation error of the $n$-th mode as $e_n(t) = z_n(t-h) - \hat{z}_n(t)$ for all $1 \leq n \leq N$ and all $t \geq h$. Defining $E^{N_0} = \begin{bmatrix} e_1 & \ldots & e_{N_0} \end{bmatrix}^\top$, the scaled error $\tilde{e}_n = \sqrt{\lambda_n} e_n$, and $\tilde{E}^{N - N_0} = \begin{bmatrix} \tilde{e}_{N_0 +1} & \ldots & \tilde{e}_{N} \end{bmatrix}^\top$, we obtain from (\ref{eq: controller 1 - Dirichlet}) and (\ref{eq: command input - Dirichlet}) that
\begin{align}
\dot{\hat{Z}}^{N_0}(t) & = A_0 \hat{Z}^{N_0}(t) + \mathfrak{B}_0 u(t-h) + LC_0 E^{N_0}(t) + L\tilde{C}_1 \tilde{E}^{N-N_0}(t) + L \zeta(t-h) \label{eq: truncated model - 4 ODEs - 1}
\end{align}
for all $t \geq h$. Defining the residue of measurement as $\zeta = \sum_{n \geq N+1} w_n \phi_n(0)$, we have $\zeta(t-h) = \sum_{n \geq N+1} w_n(t-h) \phi_n(0)$ for all $t \geq h$. The different matrices are defined by $C_0 = \begin{bmatrix} \phi_1(0) & \ldots & \phi_{N_0}(0) \end{bmatrix}$, $\tilde{C}_1 = \begin{bmatrix} \frac{\phi_{N_0 +1}(0)}{\sqrt{\lambda_{N_0 +1}}} & \ldots & \frac{\phi_{N}(0)}{\sqrt{\lambda_{N}}} \end{bmatrix}$, and $L = \begin{bmatrix} l_1 & \ldots & l_{N_0} \end{bmatrix}^\top$. Invoking the Artstein transformation (\ref{eq: controller - Dirichlet - Artstein}) and using (\ref{eq: command input - Dirichlet}) we infer that
\begin{align}
\dot{\hat{Z}}_A^{N_0}(t) & = ( A_0 + \mathfrak{B}_0 K ) \hat{Z}_A^{N_0}(t) + e^{A_0 h} LC_0 E^{N_0}(t) \label{eq: truncated model - 4 ODEs - 1 bis} \\
& \phantom{=}\; + e^{A_0 h} L\tilde{C}_1 \tilde{E}^{N-N_0}(t) + e^{A_0 h} L \zeta(t-h) \nonumber
\end{align}
for all $t \geq h$. Besides, the combination of (\ref{eq: dynamics z_n}) evaluated at time $t-h$ and (\ref{eq: truncated model - 4 ODEs - 1}) gives
\begin{equation}\label{eq: truncated model - 4 ODEs - 2}
\dot{E}^{N_0}(t) = ( A_0 - L C_0 ) E^{N_0}(t) - L \tilde{C}_1 \tilde{E}^{N-N_0}(t) - L \zeta(t-h)
\end{equation}
for all $t \geq h$.

Based on (\ref{eq: controller 2 - Dirichlet}) and defining the scaled estimation $\tilde{z}_n = \hat{z}_n / \lambda_n$ and $\tilde{Z}^{N-N_0} = \begin{bmatrix} \tilde{z}_{N_0 + 1} & \ldots & \tilde{z}_{N} \end{bmatrix}^\top$, we deduce that
\begin{equation*}
\dot{\tilde{Z}}^{N-N_0}(t) = A_1 \tilde{Z}^{N-N_0}(t) + \tilde{\mathfrak{B}}_1 u(t-h)
\end{equation*}
for $t \geq 0$ where $A_1 = \mathrm{diag}(-\lambda_{N_0+1} + q_c , \ldots , -\lambda_{N} + q_c)$ and $\tilde{\mathfrak{B}}_1 = \begin{bmatrix} \beta_{N_0 +1}/\lambda_{N_0 +1} & \ldots & \beta_N/\lambda_N \end{bmatrix}^\top$. Introducing the second Artstein tranformation:
\begin{equation}\label{eq: Dirichlet - Artstein bis}
\tilde{Z}_A^{N-N_0} (t) = e^{A_1 h}  \tilde{Z}^{N-N_0}(t) + \int_{t-h}^t e^{A_1(t-s)} \tilde{\mathfrak{B}}_1 u(s) \,\mathrm{d}s
\end{equation}
and owing to (\ref{eq: command input - Dirichlet}) we infer that
\begin{equation}\label{eq: truncated model - 4 ODEs - 3}
\dot{\tilde{Z}}_A^{N-N_0}(t) = A_1 \tilde{Z}_A^{N-N_0}(t) + \tilde{\mathfrak{B}}_1 K \hat{Z}_A^{N_0}(t) .
\end{equation}
Moreover, using (\ref{eq: dynamics z_n}) evaluated at time $t-h$ and (\ref{eq: controller 2 - Dirichlet}), the error dynamics reads
\begin{equation}\label{eq: truncated model - 4 ODEs - 4}
\dot{\tilde{E}}^{N-N_0}(t) = A_1 \tilde{E}^{N-N_0}(t)
\end{equation}
for all $t \geq h$.

Introducing the state vector
\begin{equation}\label{eq: truncated model - def X}
X = \mathrm{col}\left( \hat{Z}_A^{N_0} , E^{N_0} , \tilde{Z}_A^{N-N_0} , \tilde{E}^{N-N_0} \right) ,
\end{equation}
we obtain from (\ref{eq: truncated model - 4 ODEs - 1 bis}-\ref{eq: truncated model - 4 ODEs - 2}) and (\ref{eq: truncated model - 4 ODEs - 3}-\ref{eq: truncated model - 4 ODEs - 4}) that
\begin{equation}\label{eq: truncated model}
\dot{X}(t) = F X(t) + \mathcal{L} \zeta(t-h)
\end{equation}
for all $t \geq h$ where
\begin{equation*}
F =
\begin{bmatrix}
A_0 + \mathfrak{B}_0 K & e^{A_0 h} LC_0 & 0 & e^{A_0 h} L\tilde{C}_1 \\
0 & A_0 - L C_0 & 0 & - L\tilde{C}_1 \\
\tilde{\mathfrak{B}}_1 K & 0 & A_1 & 0 \\
0 & 0 & 0 & A_1
\end{bmatrix} , \qquad
\mathcal{L} =
\begin{bmatrix}
e^{A_0 h} L \\ -L \\ 0 \\ 0
\end{bmatrix} 
\end{equation*}
With $\tilde{X}(t) = \mathrm{col}\left( X(t) , \zeta(t-h) \right)$ and based on (\ref{eq: command input - Dirichlet}) and (\ref{eq: truncated model - 4 ODEs - 1}), we also have 
\begin{align}\label{eq: derivative v of command input u}
u(t) & = \tilde{K} X(t) , \; \forall t \geq h \;;
& v(t) & = \dot{u}(t) = K \dot{\hat{Z}}_A^{N_0}(t) , \; \forall t \geq 0  \\
& & & = E \tilde{X}(t) , \; \forall t \geq h \nonumber
\end{align}
with $\tilde{K} = \begin{bmatrix} K & 0 & 0 & 0 \end{bmatrix}$ and $E = K \begin{bmatrix} A_0 + \mathfrak{B}_0 K & e^{A_0 h} LC_0 & 0 & e^{A_0 h} L\tilde{C}_1 & e^{A_0 h} L \end{bmatrix}$.

\begin{remark}
The application of the Hautus test shows that the pairs $(A_0,\mathfrak{B}_0)$ and $(A_0,C_0)$ satisfy the Kalman condition. Hence one can always compute feedback and observer gains $K\in\R^{1 \times N_0}$ and $L\in\R^{N_0}$ so that $A_0 + \mathfrak{B}_0 K$ and $A_0 - L C_0$ are Hurwitz with arbitrary pole assignment.
\end{remark}

\begin{remark}
It is worth noting that the matrix $F$ and the vector $\mathcal{L}$ of the truncated model (\ref{eq: truncated model}) are identical to the ones obtained in \citep{lhachemi2021predictor} in the case of an input delay (instead of an output delay). However, the reduced model derived in \citep{lhachemi2021predictor} is delay-free and valid for all $t \geq 0$, essentially because the predictor components manage to completely compensate the input delay. This is not the case in the output delay setting studied in this paper due to the delayed residue of measurement $\zeta(t-h)$.  Conversely, since no input delay appears in the original PDE dynamics (\ref{eq: PDE}), the dynamics of the coefficients of projection (\ref{eq: dynamics w_n}) are delay-free. This is in contrast with the input delay setting studied in \citep{lhachemi2021predictor} where the dynamics of the modes present a time delay.
\end{remark}

\subsection{Main stability result}

We are now in position to state the main result of this section.

\begin{theorem}\label{thm1}
Let $\theta_1 \in (0,\pi/2]$, $\theta_2 \in [0,\pi/2]$, $p \in\mathcal{C}^2([0,1])$ with $p > 0$, and $\tilde{q} \in\mathcal{C}^0([0,1])$. Let $q \in\mathcal{C}^0([0,1])$ and $q_c \in\R$ be such that (\ref{eq: writting of tilde_q}) holds. Let $\delta > 0$ and $N_0 \geq 1$ be such that $-\lambda_n + q_c < - \delta$ for all $n \geq N_0 + 1$. Let $K\in\R^{1 \times N_0}\backslash\{0\}$ and $L\in\R^{N_0}$ be such that $A_0 + \mathfrak{B}_0 K$ and $A_0 - L C_0$ are Hurwitz with eigenvalues that have a real part strictly less than $-\delta<0$. Let $h > 0$ be given. For a given $N \geq N_0 +1$, assume that there exist $P \succ 0$, $\alpha>1$, and $\beta,\gamma > 0$ such that 
\begin{equation}\label{eq: thm1 - constraints}
\Theta_1 \preceq 0 ,\quad \Theta_2 \leq 0
\end{equation}
where
\begin{subequations}
\begin{align}
\Theta_1 & = \begin{bmatrix} F^\top P + P F + 2 \delta P + \alpha\gamma \Vert \mathcal{R}_N a \Vert_{L^2}^2 \tilde{K}^\top \tilde{K} & P \mathcal{L} \\ \mathcal{L}^\top P & - \beta e^{-2\delta h} \end{bmatrix} + \alpha\gamma \Vert \mathcal{R}_N b \Vert_{L^2}^2 E^\top E \label{eq: def theta1 Dirichlet} \\
\Theta_2 & = 2\gamma\left\{ - \left( 1 - \frac{1}{\alpha} \right) \lambda_{N+1}+q_c + \delta \right\} + \beta M_\phi \label{eq: def theta2 Dirichlet} 
\end{align}
\end{subequations}
where $M_\phi = \sum_{n \geq N+1} \frac{\vert \phi_n(0) \vert^2}{\lambda_n} < +\infty$. Then there exists a constant $M > 0$ such that for any initial condition $z_0 \in H^2(0,1)$ and $\hat{z}_n(0) \in\R$ so that $c_{\theta_1} z_0(0) - s_{\theta_1} z_0'(0) = 0$ and $c_{\theta_2} z_0(1) + s_{\theta_2} z_0'(1) = u_0 = K \hat{Z}_A^{N_0}(0)$, and any Lipschitz continuous $y_0 \in \mathcal{C}^0([-h,0])$ so that $y_0(0) = z_0(0)$, the trajectories of the closed-loop system composed of the plant (\ref{eq: PDE}), the delayed Dirichlet measurement (\ref{eq: system output Dirichlet}), and the controller (\ref{eq: controller - Dirichlet}-\ref{eq: command input - Dirichlet}) satisfy 
\begin{align}
& \Vert z(t,\cdot) \Vert_{H^1}^2 + \sum_{n = 1}^{N} \hat{z}_n(t)^2 \leq M e^{-2\delta t} \left( \Vert z_0 \Vert_{H^1}^2 + \sum_{n = 1}^{N} \hat{z}_n(0)^2 + u_0^2 + \Vert y_0 \Vert_\infty^2 \right) \label{eq: thm1 - stability estimate}
\end{align}
for all $t \geq 0$. Moreover, for any given $h > 0$, the constraints (\ref{eq: thm1 - constraints}) are always feasible for $N$ selected large enough.
\end{theorem}

\begin{proof} 
Let $V(t) = V_0(t)+V_1(t)$ be defined for $t \geq h$ by
\begin{subequations}\label{eq: Lyapunov function H1 norm}
\begin{align}
V_0(t) & = X(t)^\top P X(t) + \gamma \sum_{n \geq N+1} \lambda_n w_n(t)^2 \\
V_1(t) & = \beta \int_{(t-h)^+}^t e^{-2 \delta (t-s)} \zeta(s)^2 \mathrm{d}s 
\end{align}
\end{subequations}
where $(t-h)^+ = \max(t-h,0)$. The computation of the time derivative of $V$ along the system trajectories (\ref{eq: dynamics w_n}) and (\ref{eq: truncated model}) for $t > h$ gives 
\begin{align*}
& \dot{V} \leq \tilde{X}^\top \begin{bmatrix} F^\top P + P F & P \mathcal{L} \\ \mathcal{L}^\top P & - \beta e^{-2\delta h} \end{bmatrix} \tilde{X} - 2\delta V_1 + \beta \zeta^2 \\
& \quad + 2\gamma \sum_{n \geq N+1} \lambda_n (-\lambda_n + q_c) w_n^2 + 2\gamma \sum_{n \geq N+1} \lambda_n \{ a_n u + b_n v \} w_n .
\end{align*}
Using (\ref{eq: derivative v of command input u}) and invoking Young inequality, we infer for any $\alpha > 0$ that
\begin{equation*}
2 \sum_{n \geq N+1} \lambda_n a_n u w_n 
\leq \frac{1}{\alpha} \sum_{n \geq N+1} \lambda_n^2 w_n^2 + \alpha \Vert \mathcal{R}_N a \Vert_{L^2}^2 X^\top \tilde{K}^\top \tilde{K} X
\end{equation*}
and
\begin{equation*}
2 \sum_{n \geq N+1} \lambda_n b_n v w_n 
\leq \frac{1}{\alpha} \sum_{n \geq N+1} \lambda_n^2 w_n^2 + \alpha \Vert \mathcal{R}_N b \Vert_{L^2}^2 \tilde{X}^\top E^\top E \tilde{X} .
\end{equation*}
The above estimates imply 
\begin{align*}
\dot{V} + 2 \delta V 
& \leq \tilde{X}^\top \Theta_1 \tilde{X} + \beta \zeta^2 + 2\gamma \sum_{n \geq N+1} \lambda_n \left\{ - \left( 1 - \frac{1}{\alpha} \right) \lambda_n + q_c + \delta \right\} w_n^2 .
\end{align*}
Since, by definition, $\zeta = \sum_{n \geq N+1} w_n \phi_n(0)$ we obtain from Cauchy-Schwarz inequality that $\beta \zeta^2 \leq \beta M_\phi \sum_{n \geq N+1} \lambda_n w_n^2$. Hence we have that
\begin{equation}\label{eq: estimate dotV+deltaV}
\dot{V} + 2 \delta V \leq \tilde{X}^\top \Theta_1 \tilde{X} + \sum_{n \geq N+1} \lambda_n \Gamma_n w_n^2 
\end{equation}
where $\Gamma_n = 2 \gamma \left\{ - \left( 1 - \frac{1}{\alpha} \right) \lambda_n + q_c + \delta \right\} + \beta M_\phi$. Since $\alpha > 1$, we observe that $\Gamma_n \leq \Gamma_{N+1} = \Theta_2$ for all $n \geq N+1$. Owing to (\ref{eq: thm1 - constraints}) we deduce that $\dot{V} + 2 \delta V \leq 0$ for all $t > h$. Hence $V(t) \leq e^{-2\delta(t-h)}V(h)$ for all $t \geq h$.  Using standard arguments similar to the ones reported in the proof of~\cite[Thm.~6.3.3]{pazy2012semigroups} to estimate the trajectories of the closed-loop system on the time interval $[0,h]$, we infer the existence of a constant $c > 0$, independent of the initial conditions, such that 
$$\sum_{n \geq 1} \lambda_n w_n(t)^2 + \sum_{n = 1}^N \hat{z}_n(t)^2 \leq c \left( \sum_{n \geq 1} \lambda_n w_n(0)^2 + \sum_{n = 1}^N \hat{z}_n(0)^2 + u_0^2 + \Vert y_0 \Vert_\infty^2 \right)$$ 
for all $t \in [0,h]$. The claimed stability estimate (\ref{eq: thm1 - stability estimate}) is now obtained from the definition of $V$, the estimates (\ref{eq: inner product Af and f}), and by invoking the Artstein transformations (\ref{eq: controller - Dirichlet - Artstein}) and (\ref{eq: Dirichlet - Artstein bis}).

It remains to show that the constraints (\ref{eq: thm1 - constraints}) are feasible provided the dimension $N \geq N_0 + 1$ of the observer is selected sufficiently large. First, applying the Lemma reported in Appendix to the matrix $F+\delta I$, we infer for any $N \geq N_0+1$ the existence of a matrix $P \succ 0$ so that $F^\top P + P F + 2 \delta P = -I$ and $\Vert P \Vert = O(1)$ as $N \rightarrow +\infty$. Let $\alpha > 1$ be arbitrarily fixed. For any given $N \geq N_0+1$ we set $\beta = \sqrt{N} > 0$ and $\gamma = 1/N > 0$. In this case, we observe that $\Theta_2 \rightarrow -\infty$ as $N \rightarrow + \infty$ showing that $\Theta_2 \leq 0$ for $N$ large enough. Moreover, since $\Vert \tilde{K} \Vert = \Vert K \Vert$ and $\Vert \mathcal{L} \Vert$ are independent of $N$ while $\Vert P \Vert = O(1)$ and $\Vert E \Vert = O(1)$ as $N \rightarrow + \infty$, the application of the Schur complement shows that $\Theta_1 \preceq 0$ for sufficiently large $N \geq N_0 +1$. This completes the proof.  
\end{proof}

\begin{remark}
For a given $N \geq N_0 +1$ and fixing arbitrarily the value of $\alpha > 1$, the constraints (\ref{eq: thm1 - constraints}) now take the form of LMIs. Moreover, following the proof of Theorem~\ref{thm1}, this latter LMI formulation remains feasible provided $N$ is selected large enough.
\end{remark}

\section{Case of a delayed Neumann measurement}\label{sec: Neumann}

We address in this section the output feedback stabilization of the reaction-diffusion PDE described by (\ref{eq: PDE}) for $\theta_1 \in [0,\pi/2)$ with delayed Neumann measurement (\ref{eq: system output Neumann}).

\subsection{Control strategy}

Let $\delta > 0$ and $N_0 \geq 1$ be such that $-\lambda_n + q_c < - \delta < 0$ for all $n \geq N_0 + 1$. Let $N \geq N_0 + 1$ be arbitrarily fixed and that will be specified later. The observer dynamics is described for $t \geq 0$ by 
\begin{subequations}\label{eq: controller - Neumann}
\begin{align}
\hat{w}_n(t) & = \hat{z}_n(t) + b_n u(t-h) \\
\dot{\hat{z}}_n(t) & = (-\lambda_n+q_c) \hat{z}_n(t) + \beta_n u(t-h) \label{eq: controller 1 - Neumann} \\
& \phantom{=}\; - l_n \left\{ \sum_{k=1}^N \hat{w}_k(t) \phi_k'(0) - y_N(t) \right\}  ,\; 1 \leq n \leq N_0 \nonumber \\
\dot{\hat{z}}_n(t) & = (-\lambda_n+q_c) \hat{z}_n(t) + \beta_n u(t-h) ,\; N_0+1 \leq n \leq N \label{eq: controller 2 - Neumann}
\end{align}
\end{subequations}
where $l_n \in\R$ are the observer gains. The command input is then defined based on the Artstein transformation (\ref{eq: controller - Dirichlet - Artstein}) and the feedback (\ref{eq: command input - Dirichlet}). The well-posedness of the resulting closed-loop system follows the same arguments that the ones reported in Remark~\ref{rmk: WP}.

\subsection{Truncated model for stability analysis}

Proceeding as in Subsection~\ref{subsec: truncated model} while replacing the definition of $\zeta$, $\tilde{e}_n$, $C_0$ and $\tilde{C}_1$ by the following: $\zeta = \sum_{n \geq N+1} w_n \phi_n'(0)$, $\tilde{e}_n = \lambda_n e_n$, $C_0 = \begin{bmatrix} \phi_1'(0) & \ldots & \phi_{N_0}'(0) \end{bmatrix}$, and $\tilde{C}_1 = \begin{bmatrix} \frac{\phi_{N_0 +1}'(0)}{\lambda_{N_0 +1}} & \ldots & \frac{\phi_{N}'(0)}{\lambda_{N}} \end{bmatrix}$, we infer that the representation (\ref{eq: truncated model}) holds for all $t \geq h$.

\subsection{Main stability result}

We now state the main result of this section.

\begin{theorem}\label{thm2}
Let $\theta_1 \in [0,\pi/2)$, $\theta_2 \in [0,\pi/2]$, $p \in\mathcal{C}^2([0,1])$ with $p > 0$, and $\tilde{q} \in\mathcal{C}^0([0,1])$. Let $q \in\mathcal{C}^0([0,1])$ and $q_c \in\R$ be such that (\ref{eq: writting of tilde_q}) holds. Let $\delta > 0$ and $N_0 \geq 1$ be such that $-\lambda_n + q_c < - \delta$ for all $n \geq N_0 + 1$. Let $K\in\R^{1 \times N_0}\backslash\{0\}$ and $L\in\R^{N_0}$ be such that $A_0 + \mathfrak{B}_0 K$ and $A_0 - L C_0$ are Hurwitz with eigenvalues that have a real part strictly less than $-\delta<0$. Let $h > 0$ be given. For a given $N \geq N_0 +1$, assume that there exist $\epsilon\in(0,1/2]$, $P \succ 0$, $\alpha>1$, and $\beta,\gamma > 0$ such that 
\begin{equation}\label{eq: thm2 - constraints}
\Theta_1 \preceq 0 ,\quad \Theta_2 \leq 0 ,\quad \Theta_3 \geq 0
\end{equation}
where $\Theta_1$ is defined by (\ref{eq: def theta1 Dirichlet})
while
\begin{subequations}
\begin{align}
\Theta_2 & = 2\gamma\left\{ - \left( 1 - \frac{1}{\alpha} \right) \lambda_{N+1}+q_c + \delta \right\} + \beta M_\phi(\epsilon) \lambda_{N+1}^{1/2+\epsilon} \\
 \Theta_3 & = 2 \gamma \left( 1 - \frac{1}{\alpha} \right) - \frac{\beta M_\phi(\epsilon)}{\lambda_{N+1}^{1/2-\epsilon}}
\end{align}
\end{subequations}
where $M_\phi(\epsilon) = \sum_{n \geq N+1} \frac{\vert \phi'_n(0) \vert^2}{\lambda_n^{3/2+\epsilon}} < +\infty$. Then there exists a constant $M > 0$ such that for any initial condition $z_0 \in H^2(0,1)$ and $\hat{z}_n(0) \in\R$ so that $c_{\theta_1} z_0(0) - s_{\theta_1} z_0'(0) = 0$ and $c_{\theta_2} z_0(1) + s_{\theta_2} z_0'(1) = u_0 = K \hat{Z}_A^{N_0}(0)$, and any Lipschitz continuous $y_0 \in \mathcal{C}^0([-h,0])$  so that $y_0(0) = z_{0,x}(0)$, the trajectories of the closed-loop system composed of the plant (\ref{eq: PDE}), the delayed Neumann measurement (\ref{eq: system output Neumann}), and the controller composed of (\ref{eq: controller - Neumann}) and (\ref{eq: controller - Dirichlet - Artstein}-\ref{eq: command input - Dirichlet}) satisfy 
\begin{align}
& \Vert z(t,\cdot) \Vert_{H^1}^2 + \sum_{n = 1}^{N} \hat{z}_n(t)^2 \leq M e^{-2\delta t} \left( \Vert \mathcal{A} w_0 \Vert_{L^2}^2 + \sum_{n = 1}^{N} \hat{z}_n(0)^2 + u_0^2 + \Vert y_0 \Vert_\infty^2 \right) \label{eq: thm2 - stability estimate}
\end{align}
for all $t \geq 0$. Moreover, for any given $h > 0$, the constraints (\ref{eq: thm2 - constraints}) are always feasible for $N$ selected large enough.
\end{theorem}

\begin{proof}
Let $V(t) = V_0(t)+V_1(t)$ for $t \geq h$ be defined by (\ref{eq: Lyapunov function H1 norm}). The first part of the proof follows the same lines that the one of Theorem~\ref{thm1}. However, since $\zeta$ is now defined by $\zeta = \sum_{n \geq N+1} w_n \phi_n'(0)$, its estimate is replaced by the following: $\beta \zeta^2 \leq \beta M_\phi(\epsilon) \sum_{n \geq N+1} \lambda_n^{3/2+\epsilon} w_n^2$. This implies that (\ref{eq: estimate dotV+deltaV}) holds for all $t > h$ with $\Gamma_n = 2\gamma \left\{ - \left( 1 - \frac{1}{\alpha} \right) \lambda_n + q_c +\delta \right\} + \beta M_\phi(\epsilon) \lambda_{n}^{1/2+\epsilon}$. Recalling that $\epsilon\in(0,1/2]$, we obtain for any $n \geq N+1$ that $\lambda_n^{1/2+\epsilon} = \lambda_n/\lambda_n^{1/2-\epsilon} \leq \lambda_n/\lambda_{N+1}^{1/2-\epsilon}$. Since $\Theta_3 \geq 0$, this implies that $\Gamma_n \leq -\Theta_3 \lambda_n + 2\gamma \{ q_c + \delta \} \leq \Gamma_{N+1} = \Theta_2 \leq 0$ for all $n \geq N+1$. The proof of the stability estimate (\ref{eq: thm2 - stability estimate}) is now obtained using similar arguments that the ones reported in the proof of Theorem~\ref{thm1}. We use here in particular for some $\alpha_0 \in (3/4,1)$ the fact that, based on standard arguments similar to the ones reported in the proof of~\cite[Thm.~6.3.3]{pazy2012semigroups}, we have the existence of a constant $c > 0$, independent of the initial conditions, such that
$$\sum_{n \geq 1} \lambda_n^{2\alpha_0} w_n(t)^2 + \sum_{n = 1}^N \hat{z}_n(t)^2 \leq c \left( \sum_{n \geq 1} \lambda_n^{2\alpha_0} w_n(0)^2 + \sum_{n = 1}^N \hat{z}_n(0)^2 + u_0^2 + \Vert y_0 \Vert_\infty^2 \right)$$ 
for all $t \in [0,h]$. Finally, the feasibility of the constraints (\ref{eq: thm2 - constraints}) for $N$ large enough is obtained similarly by setting $\epsilon = 1/8$, $\beta = N^{1/8}$, and $\gamma = 1/N^{3/16}$. 
\end{proof}

\section{Numerical example}\label{sec: numerics}

For numerical illustration of the main results of this paper, we set the parameters $p=1$, $\tilde{q}=-5$, $\theta_1 = \pi/5$, $\theta_2 = 0$ (Dirichlet boundary control), and the input delay $h = 2\,\mathrm{s}$. The resulting reaction-diffusion PDE given by (\ref{eq: PDE}) is open-loop unstable. 

We set the feedback gain $K = -1.6037$. The observer gain is set as $L = 4.0832$ in the case of the delayed Dirichlet measurement (\ref{eq: system output Dirichlet}) while $L = 2.9666$ in the case of the delayed Neumann measurement (\ref{eq: system output Neumann}). 

With fix the prescribed decay rate $\delta = 0.5$. In the case of the Dirichlet measurement (\ref{eq: system output Dirichlet}), the constraints of Theorems~\ref{thm1} are found feasible for an observer of dimension $N = 3$, ensuring the exponential stability of the closed-loop system in $H^1$ norm. Dealing with the case of the Neumann boundary measurement (\ref{eq: system output Neumann}), the constraints of Theorem~\ref{thm2} are found feasible for an observer of dimension $N = 15$, ensuring the exponential decay of the closed-loop system in $H^1$ norm.

We complete this numerical illustration by depicting the closed-loop system behavior in the case of the delayed Dirichlet measurement (\ref{eq: system output Dirichlet}). We set the initial conditions $z_0(x)=5x^2(x-3/4)$ and $y_0(\tau) = 3\cos(10\pi(\tau+h))\sin(3\pi\tau)$ for $\tau \leq 0$, while $\hat{z}_n(0)$ are fixed so that $u_0 = K \hat{Z}_A^{N_0}(0)$. In this setting, the time domain evolution of the closed-loop system is depicted in Fig.~\ref{fig: sim CL}. As predicted by Theorem~\ref{thm1}, we observe the exponential decay of the both state of the PDE and observation error in spite of the $h = 2\,\mathrm{s}$ output delay.

\begin{figure}
     \centering
     	\subfigure[State of the reaction-diffusion system $z(t,x)$]{
		\includegraphics[width=3.5in]{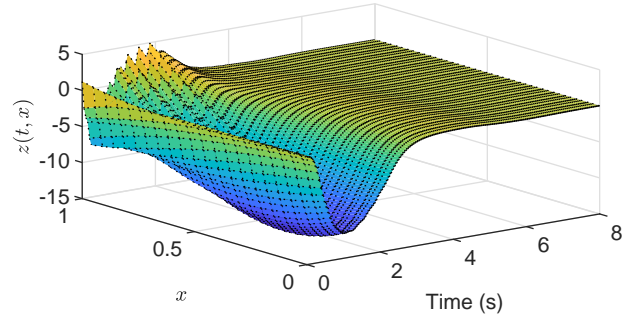}
		}
     	\subfigure[Error of observation $e(t,x) = z(t-h,x) - \hat{z}(t,x)$ for $t \geq h$]{
		\includegraphics[width=3.5in]{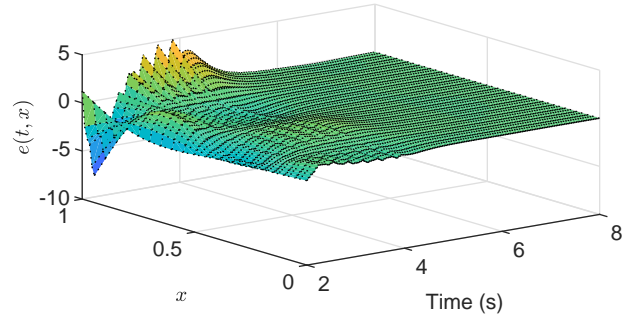}
		}
     	\subfigure[Delayed measurement $y_D(t) = z(t-h,0)$ with delay $h = 2\,\mathrm{s}$]{
		\includegraphics[width=3.5in]{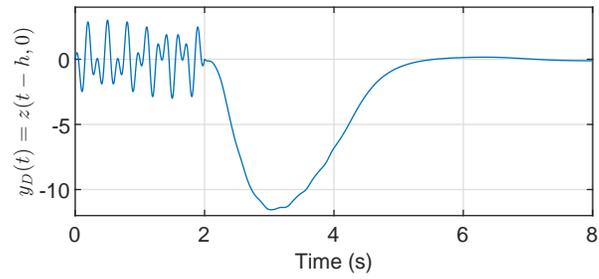}
		\label{fig: sim CL - input}
		}		
     \caption{Time evolution of the closed-loop system for Dirichlet measurement $y_D (t) = z(t-h,0)$ with delay $h = 2\,\mathrm{s}$}
     \label{fig: sim CL}
\end{figure}

\section{Extension to input and output delays}\label{sec: extension}

We briefly discuss in this section how the output feedback boundary stabilization of general 1-D reaction-diffusion PDEs in the presence of both input and output delays can be achieved by merging the techniques developed in this paper for an output delay and the ones reported in \citep{lhachemi2021predictor} for an input delay. We focus the presentation on the Dirichlet measurement but the same procedure can be used to address the case of the Neumann measurement. Consider the reaction-diffusion system with boundary control described by
\begin{subequations}\label{eq: PDE - extension}
\begin{align}
& z_t(t,x) = (p(x) z_x(t,x))_x - \tilde{q}(x) z(t,x) \\
& c_{\theta_1} z(t,0) - s_{\theta_1} z_x(t,0) = 0 \\
& c_{\theta_2} z(t,1) + s_{\theta_2} z_x(t,1) = u(t-h_i) \\
& z(0,x) = z_0(x) 
\end{align}
\end{subequations}
for $t > 0$ and $x \in (0,1)$. The different parameters are defined as in (\ref{eq: PDE}) while $h_i > 0$ is an input delay. We assume that $u(\tau)=0$ for all $\tau \leq 0$. In the case case $\theta_1 \in (0,\pi/2]$, the delayed Dirichlet boundary measurement is defined by 
\begin{equation}\label{eq: system output Dirichlet - extension}
y_D(t) = 
\left\{
\begin{array}{ll}
z(t-h_o,0) , & t \geq h_o \\
y_0(t-h_o) , & 0 \leq t \leq h_o
\end{array}
\right.
\end{equation}
with output delay $h_o > 0$. We introduce $q \in\mathcal{C}^0([0,1])$ and $q_c \in\R$ so that (\ref{eq: writting of tilde_q}) holds. Defining the change of variable
\begin{equation}\label{eq: change of variable - extension}
w(t,x) = z(t,x) - \frac{x^2}{c_{\theta_2} + 2 s_{\theta_2}} u(t-h_i) . 
\end{equation}
we infer that
\begin{equation}\label{eq: link z_n and w_n - extension}
w_n(t) = z_n(t) + b_n u(t-h_i), \quad n \geq 1 .
\end{equation}
The projections of the PDE in $z$ coordinates gives
\begin{equation}\label{eq: dynamics z_n - extension}
\dot{z}_n(t) = (-\lambda_n + q_c) z_n(t) + \beta_n u(t-h_i)
\end{equation}
while, in $w$ coordinates,
\begin{subequations}\label{eq: dynamics w_n - extension}
\begin{align}
\dot{u}(t) & = v(t) \\
\dot{w}_n(t) & = (-\lambda_n + q_c) w_n(t) + a_n u(t-h_i) + b_n v(t-h_i) .
\end{align}
\end{subequations}
The delayed measurement is expressed for $t \geq h_o$ by
\begin{equation}\label{eq: system output - series - extension}
y_D(t) = \sum_{n \geq 1} w_n(t-h_o) \phi_n(0) . 
\end{equation}

Let $\delta > 0$ and $N_0 \geq 1$ be such that $-\lambda_n + q_c < - \delta < 0$ for all $n \geq N_0 + 1$. Let $N \geq N_0 + 1$ be arbitrarily given. The observer dynamics, used to estimate the $N$ first modes $z_n(t-h_o)$ of the PDE at time $t-h_o$, is described for $t \geq 0$ by 
\begin{subequations}\label{eq: controller - Dirichlet - extension}
\begin{align}
\hat{w}_n(t) & = \hat{z}_n(t) + b_n u(t-h_{io}) \\
\dot{\hat{z}}_n(t) & = (-\lambda_n+q_c) \hat{z}_n(t) + \beta_n u(t-h_{io}) \label{eq: controller 1 - Dirichlet - extension} \\
& \phantom{=}\; - l_n \left\{ \sum_{k=1}^N \hat{w}_k(t) \phi_k(0) - y_D(t) \right\}  ,\; 1 \leq n \leq N_0 \nonumber \\
\dot{\hat{z}}_n(t) & = (-\lambda_n+q_c) \hat{z}_n(t) + \beta_n u(t-h_{io}) ,\; N_0+1 \leq n \leq N \label{eq: controller 2 - Dirichlet - extension}
\end{align}
\end{subequations}
where $l_n \in\R$ are the observer gains and $h_{io} = h_i + h_o > 0$. Introducing the predictor component defined by
\begin{equation}\label{eq: controller - Dirichlet - Artstein - extension}
\hat{Z}_A^{N_0} (t) = e^{A_0 h_{io}}  \hat{Z}^{N_0}(t) + \int_{t-{h_{io}}}^t e^{A_0(t-s)} \mathfrak{B}_0 u(s) \,\mathrm{d}s ,
\end{equation}
we define the control input as 
\begin{equation}\label{eq: command input - Dirichlet - extension}
u(t) = K \hat{Z}_A^{N_0}(t) 
\end{equation}
for all $t \geq 0$ where $K\in\R^{1 \times N_0}$ is the feedback gain. Then proceeding as in Subsection~\ref{subsec: truncated model} but with $\tilde{Z}_A^{N-N_0}$ defined by 
\begin{equation}\label{eq: Dirichlet - Artstein bis - extension}
\tilde{Z}_A^{N-N_0} (t) = e^{A_1 h_{io}}  \tilde{Z}^{N-N_0}(t) + \int_{t-h_{io}}^t e^{A_1(t-s)} \tilde{\mathfrak{B}}_1 u(s) \,\mathrm{d}s ,
\end{equation}
we infer that
\begin{equation}\label{eq: truncated model - extension}
\dot{X}(t) = F X(t) + \mathcal{L} \zeta(t-h_o)
\end{equation}
for all $t \geq h_0$ where $X$ is defined by (\ref{eq: truncated model - def X}) and $\zeta = \sum_{n \geq N+1} w_n \phi_n(0)$ while the matrix $F$ and the vector $\mathcal{L}$ are defined by
\begin{equation*}
F =
\begin{bmatrix}
A_0 + \mathfrak{B}_0 K & e^{A_0 h_{io}} LC_0 & 0 & e^{A_0 h_{io}} L\tilde{C}_1 \\
0 & A_0 - L C_0 & 0 & - L\tilde{C}_1 \\
\tilde{\mathfrak{B}}_1 K & 0 & A_1 & 0 \\
0 & 0 & 0 & A_1
\end{bmatrix} , \qquad
\mathcal{L} =
\begin{bmatrix}
e^{A_0 h_{io}} L \\ -L \\ 0 \\ 0
\end{bmatrix} .
\end{equation*}
Defining $\tilde{X}(t) = \mathrm{col}\left( X(t) , \zeta(t-h_o) \right)$, we have $u(t) = \tilde{K} X(t)$ and $v(t) = \dot{u}(t) = K \dot{\hat{Z}}_A^{N_0}(t)$ for all $t \geq 0$ where $\tilde{K} = \begin{bmatrix} K & 0 & 0 & 0 \end{bmatrix}$. Moreover, we also have $\dot{\hat{Z}}_A^{N_0}(t) = E \tilde{X}(t)$ for all $t \geq h_o$ with $E = \begin{bmatrix} A_0 + \mathfrak{B}_0 K & e^{A_0 h_{io}} LC_0 & 0 & e^{A_0 h_{io}} L\tilde{C}_1 & e^{A_0 h_{io}} L \end{bmatrix}$. 

Combining now the approaches developed in this paper and in \citep{lhachemi2021predictor} to handle the output delay $h_o > 0$ appearing in (\ref{eq: truncated model - extension}) and the input delay $h_i > 0$ occurring in (\ref{eq: dynamics w_n - extension}), respectively, we arrive at the following theorem.

\begin{theorem}\label{thm3}
Let $\theta_1 \in (0,\pi/2]$, $\theta_2 \in [0,\pi/2]$, $p \in\mathcal{C}^2([0,1])$ with $p > 0$, and $\tilde{q} \in\mathcal{C}^0([0,1])$. Let $q \in\mathcal{C}^0([0,1])$ and $q_c \in\R$ be such that (\ref{eq: writting of tilde_q}) holds. Let $\delta > 0$ and $N_0 \geq 1$ be such that $-\lambda_n + q_c < - \delta$ for all $n \geq N_0 + 1$. Let $K\in\R^{1 \times N_0}$ and $L\in\R^{N_0}$ be such that $A_0 + \mathfrak{B}_0 K$ and $A_0 - L C_0$ are Hurwitz with eigenvalues that have a real part strictly less than $-\delta<0$. Let $h_i,h_o > 0$ be given. For a given $N \geq N_0 +1$, assume that there exist $P \succ 0$, $Q_1,Q_2 \succeq 0$, $\alpha>1$, and $\beta,\gamma > 0$ such that 
\begin{equation}\label{eq: thm3 - constraints}
\Theta_1 \preceq 0 ,\quad \Theta_2 \leq 0 ,\quad R_1 \preceq 0 ,\quad R_2 \preceq 0
\end{equation}
where
\begin{subequations}
\begin{align}
\Theta_1 & = \begin{bmatrix} F^\top P + P F + 2 \delta P + \tilde{Q}_1 & P \mathcal{L} \\ \mathcal{L}^\top P & - \beta e^{-2\delta h_o} \end{bmatrix} + E^\top Q_2 E  \\
\Theta_2 & = 2\gamma\left\{ - \left( 1 - \frac{1}{\alpha} \right) \lambda_{N+1}+q_c + \delta \right\} + \beta M_\phi \\
R_1 & = - e^{-2\delta h_{i}} Q_1 + \alpha\gamma \Vert \mathcal{R}_N a \Vert_{L^2}^2 K^\top K \\
R_2 & = - e^{-2\delta h_{i}} Q_2 + \alpha\gamma \Vert \mathcal{R}_N b \Vert_{L^2}^2 K^\top K
\end{align}
\end{subequations}
where $\tilde{Q}_1 = \mathrm{diag}(Q_1,0,0,0)$ and $M_\phi = \sum_{n \geq N+1} \frac{\vert \phi_n(0) \vert^2}{\lambda_n} < +\infty$. Then there exists a constant $M > 0$ such that for any initial condition $z_0 \in H^2(0,1)$ and $\hat{z}_n(0) \in\R$ so that $c_{\theta_1} z_0(0) - s_{\theta_1} z_0'(0) = 0$ and $c_{\theta_2} z_0(1) + s_{\theta_2} z_0'(1) = 0$, and any Lipschitz continuous $y_0 \in \mathcal{C}^0([-h_o,0])$ so that $y_0(0) = z_0(0)$, the trajectories of the closed-loop system composed of the plant (\ref{eq: PDE - extension}), the delayed Dirichlet measurement (\ref{eq: system output Dirichlet - extension}), and the controller (\ref{eq: controller - Dirichlet - extension}-\ref{eq: command input - Dirichlet - extension}) with zero control in negative time and zero initial condition for the observer ($u(\tau) = 0$ for $\tau < 0$ and $\hat{z}_n(0) = 0$) satisfy 
\begin{align}
\Vert z(t,\cdot) \Vert_{H^1}^2 + \sum_{n = 1}^{N} \hat{z}_n(t)^2 \leq M e^{-2\delta t} \left( \Vert z_0 \Vert_{H^1}^2 + \Vert y_0 \Vert_\infty^2 \right) \label{eq: thm3 - stability estimate}
\end{align}
for all $t \geq 0$. Moreover, for any given $h_i,h_o > 0$, the constraints (\ref{eq: thm3 - constraints}) are always feasible for $N$ selected large enough.
\end{theorem}

\section{Conclusion}\label{sec: conclusion}
This paper solved the problem of output feedback stabilization of 1-D reaction-diffusion PDEs in the presence of an arbitrary output delay. The proposed setting embraces general Dirichlet/Neumann/Robin boundary condition/control along with Dirichlet/Neumann boundary measurement. While the output feedback stabilization of general 1-D reaction-diffusion PDEs was reported in \citep{lhachemi2021predictor} for an arbitrary input delay and in \citep{lhachemi2021boundary} for an arbitrary state delay in the reaction term, the results presented in this paper complete the full picture by addressing the case of an arbitrary output delay. Furthermore, we showed how the combination of the techniques developed in this paper with the ones reported in \citep{lhachemi2021predictor} allows to address the case of simultaneous input and output delays.

It is worth noting that the main results of this paper can be extended in a straightforward manner to any $\theta_1,\theta_2\in[0,\pi)$. This can be achieved by 1) selecting $q$ in (\ref{eq: writting of tilde_q}) sufficiently large positive so that (\ref{eq: inner product Af and f}) holds true; 2) adapt the change of variable formula (\ref{eq: change of variable}) to avoid a possible division by 0 by using $w(t,x) = z(t,x) - \frac{x^\alpha}{c_{\theta_2} + \alpha s_{\theta_2}} u(t)$ where $\alpha > 1$ is selected such that $c_{\theta_2} + \alpha s_{\theta_2} \neq 0$.

\section*{Funding}

The work of C. Prieur has been partially supported by MIAI@Grenoble Alpes (ANR-19-P3IA-0003)

\section{References}

\bibliographystyle{apacite}
\bibliography{interactapasample}

\appendix

\section{Technical lemma}

The following Lemma is an immediate generalization of the result presented in \citep{katz2020constructive}.

\begin{lemma}\label{lem: useful lemma}
Let $n,m,N \geq 1$, $M_{11} \in \R^{n \times n}$ and $M_{22} \in \R^{m \times m}$ Hurwitz, $M_{12} \in \R^{n \times m}$, $M_{14}^N \in\R^{n \times N}$, $M_{24}^N \in\R^{m \times N}$, $M_{31}^N \in\R^{N \times n}$, $M_{33}^N,M_{44}^N \in \R^{N \times N}$, and
\begin{equation*}
F^N = \begin{bmatrix}
M_{11} & M_{12} & 0 & M_{14}^N \\
0 & M_{22} & 0 & M_{24}^N \\
M_{31}^N & 0 & M_{33}^N & 0 \\
0 & 0 & 0 & M_{44}^N
\end{bmatrix} .
\end{equation*}
We assume that there exist constants $C_0 , \kappa_0 > 0$ such that $\Vert e^{M_{33}^N t} \Vert \leq C_0 e^{-\kappa_0 t}$ and $\Vert e^{M_{44}^N t} \Vert \leq C_0 e^{-\kappa_0 t}$ for all $t \geq 0$ and all $N \geq 1$. Moreover, we assume that there exists a constant $C_1 > 0$ such that $\Vert M_{14}^N \Vert \leq C_1$, $\Vert M_{24}^N \Vert \leq C_1$, and $\Vert M_{31}^N \Vert \leq C_1$ for all $N \geq 1$. Then there exists a constant $C_2 > 0$ such that, for any $N \geq 1$, there exists a symmetric matrix $P^N \in\R^{n+m+2N}$ with $P^N \succ 0$ such that $P^N F^N + (F^N)^\top P^N = - I$ and $\Vert P^N \Vert \leq C_2$.
\end{lemma}

\end{document}